\documentclass[11pt]{article}
\usepackage{amsmath}
\usepackage{amssymb,amsbsy,amsthm}
\usepackage{graphicx}
\usepackage[dvipsnames]{xcolor}
\usepackage{enumerate}
\usepackage{cases}
\usepackage[margin=1in]{geometry}
\usepackage{pdfsync}
\usepackage{hyperref}
\usepackage{wrapfig} % For figures on the side within the text

\allowdisplaybreaks[1]

\numberwithin{equation}{section}

\DeclareMathOperator{\R}{\mathbb{R}} % Real numbers
\DeclareMathOperator{\N}{\mathbb{N}} % Natural numbers
\newcommand{\p}{\mathbb{P}} % Probability P
\newcommand{\E}{\mathbb{E}} % Expectation E
\newcommand{\eps}{\varepsilon} % Epsilon
\def\cC{{\mathcal C}}

\def\cX{{\mathcal X}}

\newtheorem{theorem}{Theorem}[section]

\newtheorem{proposition}[theorem]{Proposition}

\newtheorem{fact}[theorem]{Fact}
\newtheorem{fact*}[]{Fact}
\newtheorem{corollary}[theorem]{Corollary}

\newtheorem{question}[theorem]{Question}
\newtheorem{definition}[theorem]{Definition}
\newtheorem{definition*}[]{Definition}

\begin{document}

\title{Intersections of random sets}
\author{
	Jacob Richey
	\and
	Amites Sarkar
}
\date{\today}

\maketitle

%%%%%%%%%%%%%%%%
%%% Abstract %%%
%%%%%%%%%%%%%%%%

\begin{abstract}
We consider a variant of a classical coverage process, the boolean model in $\mathbb{R}^d$. Previous efforts have focused on convergence of the unoccupied region containing the origin to a well studied limit $C$. We study the intersection of sets centered at points of a Poisson point process confined to the unit ball. Using a coupling between the intersection model and the original boolean model, we show that the scaled intersection converges weakly to the same limit $C$. Along the way, we present some tools for studying statistics of a class of intersection models.
%\todo{Thoughts on the abstract? It feels a bit passive, maybe a bit too minimal?}
\end{abstract}

%%%%%%%%%%%%%%%%
%%% Document %%%
%%%%%%%%%%%%%%%%

%%%%%%%%%%%%%%%%%%%%%%%%%%%%%%%%%%%%%%%%%%%%
\section{Introduction} \label{sec:intro} %%%
%%%%%%%%%%%%%%%%%%%%%%%%%%%%%%%%%%%%%%%%%%%%

The Boolean model is a geometric model for random sets in space: for example, comets striking the moon, or antibodies attaching to a virus. One starts with a Poisson point process $\mathcal{X}$ of intensity $\lambda$ on $\mathbb{R}^d$ (see Appendix A for some background on Poisson processes), and, for simplicity, some fixed set $S$. The Boolean model refers to the set

\begin{equation}
B_{\lambda, S} = \bigcup_{X \in \mathcal{X}} (X + S),
\end{equation}

\noindent where, for $x \in \mathbb{R}^d$,

\begin{equation}
x+S = \{x + y: y \in S\}.
\end{equation}

\noindent Conditioning on the (positive probability) event $\{0 \notin B_{\lambda, S}\}$, one then studies, among other things, the uncovered component of $\R^d$ containing $0$, known as the {\it Crofton cell}.

\medskip

The Boolean model was first defined and studied in the 1960s by E.N. Gilbert, who investigated both percolation~\cite{Gil61} and coverage~\cite{Gil65}. Gilbert's results (from both papers) were improved by Hall~\cite{Hal85a,Hal85b}, and Hall's result on coverage~\cite{Hal85a} was improved and generalized by Janson~\cite{Jan} and Molchanov~\cite{Mol}. The books by Meester and Roy~\cite{MeRo} and Hall~\cite{Hallbook} are standard references for the percolation and coverage processes respectively, and the more recent books by Haenggi~\cite{Martin} and Franceschetti and Meester~\cite{Massimo} cover applications to wireless networks (the motivating example from~\cite{Gil61}).

\medskip

There are a number of well-studied cousins of the Boolean model. One such class of models is tessellations. For example, consider a hyperplane tessellation of intensity $\lambda$ on $\R^d$. Namely, let $\mathcal{R} = \{R_n: n \in \mathbb{N}\}$ denote a Poisson point process with intensity $\lambda$ -- or, more generally, with intensity measure $\nu$ -- on $\mathbb{R},$ and let $\theta_n, n \in \N$, be iid uniform over the $d-1$ dimensional unit sphere, independent of $\mathcal{R}$. The corresponding Poisson field $\mathcal{H}$ refers to the set of hyperplanes $H_n$ given by

\begin{equation} \label{pht_def}
H_n = \{x \in \R^d: x \cdot \theta_n = R_n\},
\end{equation}

\noindent i.e., $H_n$ is the hyperplane passing through the point $R_n\theta_n$, with normal vector parallel to $\theta_n$. Note that, on average, $2 \lambda$ hyperplanes intersect the unit sphere in this model. The analogous object in this context is the connected component of the tessellation containing 0, denoted $D_0^\lambda$.

\medskip

Hyperplane tessellations have a history dating back to the 19th century. Crofton's formula (1868) expresses the length of a curve $C$ in terms of the expected number of times $C$ meets a random line, i.e., the expected number of times $C$ intersects a random line tessellation. Generalizations of Crofton's formula led to the development of Integral Geometry~\cite{San}. Meanwhile, inspired by a question of Niels Bohr on cloud chamber tracks, Goudsmit in 1945~\cite{Goud} computed the variance of the cell size in a random line tessellation, and Miles in 1964~\cite{Miles1, Miles2} initiated the systematic study of random line tessellations (with potential applications to papermaking in mind). We will make use of Goudsmit's result later. Calka's recent survey article~\cite{Calka} gives an excellent overview of the field.

\medskip

We will also use a slightly different tessellation model: tessellation by spheres, rather than hyperplanes. Let $\mathcal{X}$ be a Poisson point process of intensity $\lambda$ on $\mathbb{R}^d$, and let $\mathbb{S} = \partial \mathbb{B} = \{x \in \mathbb{R}^d: ||x||_2 = 1\}$ denote the boundary of the unit ball $\mathbb{B}$ in $\mathbb{R}^d$. Then the set

\begin{equation}
\bigcup_{X \in \mathcal{X}} (X + \mathbb{S})
\end{equation}

\noindent partitions $\mathbb{R}^d$ into a union of disjoint connected components. As before, it is common to study the connected component of this tessellation containing the origin, denoted $E_0^\lambda$.

A result due to Michel and Paroux~\cite{MP} (see also~\cite{CMP}) gives a common scaling limit for these sets:

\begin{theorem}\label{CMP} Let $C_0^\lambda$, $D_0^\lambda$ and $E_0^\lambda$ denote the Crofton cell in the boolean model (with $S=\mathbb{B}$), and the connected components containing the origin in the hyperplane and sphere tessellation models, respectively. These three models have a common scaling limit in distribution:

\begin{equation}
\lim_{\lambda \to \infty} S_d\lambda C_0^\lambda \sim \lim_{\lambda \to \infty} 2\lambda D_0^\lambda  \sim \lim_{\lambda \to \infty} 2S_d\lambda E_0^\lambda \sim 2C,
\end{equation}

\noindent where $S_d$ is the surface area measure of $\mathbb{S} = \partial \mathbb{B}$, $C$ is the law of the normalized Crofton cell $D_0^1$, and $\sim$ denotes equality in distribution. \end{theorem}

\eject

The scaling in this theorem is nontrivial and deserves some comment. Consider first the tessellation models. As $\lambda\to\infty$, the expected volume of the connected component containing the origin is inversely proportional to the expected number of connected components in $B_{\eps} = \{x:||x|| < \eps \}$. This latter quantity is in turn proportional to the number of points of intersection (of spheres or hyperplanes) in $B_{\eps}$. The number of such intersection points scales as $\lambda^d$, since each intersection point is determined by $d$ spheres (respectively hyperplanes), so the expected volume of the Crofton cell scales as $\lambda^{-d}$, leading to a linear scale factor of $1/\lambda$. Finally, the constants can be determined by computing the expected number of spheres (respectively hyperplanes) meeting $B_{\eps}$ in each of the three models: for $C_0^\lambda$, $D_0^\lambda$ and $E_0^\lambda$ these are asymptotically $S_d\lambda\eps,2\lambda\eps$ and $2S_d\lambda\eps$ respectively.

Our contribution is to introduce a new model to this group, which we term the \emph{boolean intersection model}. Let $\mathbb{B}$ denote the closed unit ball in $\mathbb{R}^d$, and let $\mathcal{C}_\lambda$ be a Poisson point process with intensity $\lambda$ on $\mathbb{B}$. (We will focus on the uniform intensity case, and discuss in Section \ref{sec:prelim} some ideas that apply for general intensity measures). Let

\begin{equation} \label{intersection_def_ball}
I^\lambda = \mathbb{B} \cap \left( \bigcap_{C \in \mathcal{C}_\lambda} (C + \mathbb{B}) \right)
\end{equation}

\noindent be the intersection of copies of $\mathbb{B}$ shifted by $C \in \mathcal{C}_{\lambda}$. By definition, set $I^\lambda = \mathbb{B}$ if $\mathcal{C}_{\lambda} = \emptyset$. Note that for each $\lambda$, $I^\lambda$ is a convex (connected) set contained in $\mathbb{B}$, containing $0$. Our goal will be to show that $I^\lambda$ shares the same limit as the Boolean and Poisson hyperplane tessellation processes, up to a constant.

\begin{theorem}\label{intersection_conv} Let $C$ denote the law of the normalized Crofton cell $D_0^1$. The following distributional convergence holds:

\begin{equation} \label{main_conv}
\lim_{\lambda \to \infty} S_d\lambda I^\lambda \sim 2C.
\end{equation}

\end{theorem}

Write $|I|$ for the Lebesgue measure of a measurable set $I \subset \mathbb{R}^d$, and $\omega_d = $ vol$(\mathbb{B}) = \frac{1}{d} S_d$.

\begin{proposition}\label{volume_conv} The limiting volume of the intersection is given by

\begin{equation}
\lim_{\lambda \to \infty} \lambda^d\E |I^\lambda| = \frac{d! \omega_d}{\omega_{d-1}^d}.
\end{equation}

\noindent For example, when $d=2$, $\lim_{\lambda \to \infty} \lambda^2\E |I^\lambda| = \frac{\pi}{2}.$ \end{proposition}

One proof of this fact is via (\ref{main_conv}), and an appeal to classical results on the statistics of the Crofton cell $C$~\cite{Goud}; we outline the details in Section \ref{sec:ccell}. We give an alternative proof based on general results for the intersection model in Section \ref{sec:prelim}. Similar formulas hold for a more general class of boolean intersection models, namely where the underlying point process is not uniform, or the balls are exchanged for other compact sets. For simplicity we focus on the case of balls of fixed radius, and comment on generalizations in Section \ref{general_intmodel}. Additionally, one can formulate all these results using `fixed count' versions of these models: that is, instead of using an underlying Poisson process, for each $n$ define $I_n$ as the intersection of $n$ iid translations of $\mathbb{B}$. This model
%, which we analyze briefly in Section \ref{sec:prelim},
also shares the Crofton cell $D_0^1$ as a scaling limit.

\subsection{Random variables in the space of closed sets}

Our main result, Theorem \ref{intersection_conv}, refers to a convergence of random sets. The usual modes of convergence -- in law, in probability, and in $L^p$ -- are trickier to define for random sets. We follow the approach taken in~\cite{CMP} and most of the existing literature, namely via the Hausdorff metric $d_H$ on the space $E$ of closed subsets of $\mathbb{R}^d$. Namely, if $X_n$ and $X$ are defined on the same probability space, we say $X_n \to X$ if for every $\epsilon > 0$,

\begin{equation}\label{prob_conv_def}
\p(d_H(X_n, X) > \epsilon) \to 0 \hspace{5pt} \text{ as } n \to \infty.
\end{equation}

\noindent Our proof of Theorem \ref{intersection_conv} gives an explicit coupling between the intersection model and the Crofton cell, under which (\ref{prob_conv_def}) holds.

%%%%%%%%%%%%%%%%%%%%%%%%%%%%%%%%%%%%%%%%%%%%
\section{Overview of results} \label{sec:outline} %%%
%%%%%%%%%%%%%%%%%%%%%%%%%%%%%%%%%%%%%%%%%%%%

Although our main result is Theorem \ref{intersection_conv}, we will present two different approaches to its corollary, Proposition \ref{volume_conv}. The first is direct, and appears in Section \ref{sec:prelim}. The second is via Theorem \ref{intersection_conv}, and occupies Sections \ref{sec:lemmas}, \ref{sec:coupling} and \ref{sec:ccell}. Below, we summarize each approach in turn.

In Section \ref{sec:prelim}, we first consider a single ball $B$, whose center lies inside the unit ball $\mathbb{B}$. A simple volume calculation lets us estimate (in Proposition \ref{G_formula}) the probability that $B$ contains the line segment $[0,r]$ on the positive $x$-axis. Next, we consider the full model, in which unit balls are centered at points of a Poisson process of intensity $\lambda$ inside $\mathbb{B}$, forming an intersection $I^{\lambda}$. Write $R^{\lambda}$ for the radius of $I^{\lambda}$ in the positive $x$ direction (for instance). Proposition \ref{G_formula}, together with an approximation argument, allows us to show that the suitably scaled radius $R^{\lambda}$ converges in distribution to an exponential random variable of mean one. Finally, Proposition \ref{area_calc} relates the expected volume of $I^{\lambda}$ to the $d$th moment of $R^{\lambda}$, completing the calculation of $\E |I^{\lambda}|$ and proving Proposition \ref{volume_conv}.

As part of this analysis, we also prove a parallel series of results for a related hyperplane model $K^{\lambda}$, which approximates the ball model $I^{\lambda}$ near the origin.

Section \ref{sec:lemmas} contains some probabilistic and geometric lemmas that we will use in Section \ref{sec:coupling}.

In Section \ref{sec:coupling}, we prove Theorem \ref{intersection_conv}. The main idea is to couple the intersection model with a sphere tessellation model. With $I^{\lambda}$ as above, and $J^{\lambda}$ defined as the cell containing the origin in an appropriately scaled sphere tessellation model, we show that, in our coupling, the scaled Hausdorff distance between $I^{\lambda}$ and $J^{\lambda}$ converges to 0 almost surely as $\lambda \to \infty$. To achieve this, we carry out the following steps:

\begin{enumerate} \item Show that only points of the Poisson process close to the boundary -- namely, within distance $\eps_\lambda = \frac{\log^2 \lambda}{\lambda}$ of $\partial \mathbb{B}$ -- contribute to the shape of $I^\lambda$ (asymptotically almost surely as $\lambda \to \infty$). We use a coupon-collector type process on $\partial \mathbb{B}$ as a test: if there are points of $\mathcal{C}_\lambda$ close to $\partial \mathbb{B}$ that are ``spread out enough" (in terms of angle), then $I^\lambda$ must be contained within a small ball at the origin.

\item Match spheres close to $\partial \mathbb{B}$ from the intersection model with spheres close to $\partial \mathbb{B}$ from the tessellation model. Since some of the sphere centers in the tessellation model land outside $\mathbb{B}$, while the sphere centers in the intersection model are confined to $\mathbb{B}$, we must match them by shifting. We accordingly identify each point $R \theta$ of the underlying tessellation process (with $R \in (1, 1 + \eps_\lambda)$ and $\theta \in \partial \mathbb{B}$) with the point $(R-2)\theta$ in the intersection process. The spheres centered at these two points have almost the same effect on $J^\lambda$, except that one makes a `concave' boundary and the other a `convex' boundary.

 \item Show that the error in carrying out Step 2, in terms of the Hausdorff distance between the suitably scaled sets $I^\lambda$ and $J^{\lambda}$, is negligible. This involves a careful coupling between the underlying point processes: since we scale everything by a factor of $\lambda$, it is necessary to have an error of order $o(\lambda^{-1}),$ not just $o(1)$.

\end{enumerate}

Finally, in Section \ref{sec:ccell}, we derive Proposition \ref{volume_conv} from Theorem \ref{intersection_conv}, using the $d$-dimensional version of Goudsmit's 1945 calculation~\cite{Goud} for the area of the Crofton cell. Although the relevant formulas appear in the work of Matheron~\cite{Math} (Chapter 6), our treatment is short and self contained.

%%%%%%%%%%%%%%%%%%%%%%%%%%%%%%%%%%%%%%%%%%%%
\section{Statistics for the intersection model} \label{sec:prelim} %%%
%%%%%%%%%%%%%%%%%%%%%%%%%%%%%%%%%%%%%%%%%%%%

To warm up, consider the following generic intersection process on $\R$. Let $\mathcal{C}_\lambda$ be a Poisson process with intensity $\lambda$ on $[-1,1]$, and define

\begin{equation}
U = U_{\lambda} = [-1,1] \cap \left( \bigcap_{c \in C_{\lambda}} c + [-1,1] \right)
\end{equation}

\noindent where $c + [-1,1] = [c-1, c+1]$ for $c \in \R$, with the convention that $U = [-1,1]$ if $\mathcal{C}_{\lambda} = \emptyset$. $U$ is always an interval, so it is determined by $\sup U$ and $\inf U$. The distribution of the infimum can be determined as follows. We have that $\inf U < -u$ if and only if $\max \mathcal{C}_{\lambda} < 1-u$, which in turn holds if and only if each element of $\mathcal{C}_{\lambda}$ is less than $1-u$, i.e., if no point of $\mathcal{C}_{\lambda}$ lies in the interval $[1-u,1]$. The latter event has probability $e^{-\lambda u}$, so $-\inf U$ is approximately distributed as $\text{Exp}(\lambda^{-1})$ as $\lambda \to \infty$.

\noindent Note that $\inf U$ is a function of $\mathcal{C}_{\lambda} \cap [0,1]$ and $\sup U$ is a function of $\mathcal{C}_{\lambda} \cap [-1,0]$; basic properties of Poisson processes imply that $\inf U$ and $\sup U$ are independent. Thus, as $\lambda \to \infty$, the left and right endpoints of $U$ shrink to 0 independently, and 

\begin{equation}
|U_{\lambda}| \approx \text{Exp}(\lambda^{-1}) + \text{Exp}(\lambda^{-1}) \sim \text{Gamma}(2, \lambda^{-1}).
\end{equation}

%%%%%%%%%%%%%%%%%%%%%%%%%%%%%%%%
\subsection{Ball and hyperplane intersection models}
%%%%%%%%%%%%%%%%%%%%%%%%%%%%%%%%

We now explore some basic properties of two intersection models: the ball intersection model $I^\lambda$ (defined as in (\ref{intersection_def_ball})), and an analogous intersection model with hyperplanes. Let $\widetilde{\cC_\lambda}$ be a Poisson point process on $\mathbb{B}$ with intensity measure $\lambda (\nu \times \sigma)$, where $\sigma$ is surface-area measure on $\mathbb{S}$ and $\nu$ is the measure supported on $[0,1]$ given by

\begin{equation}
\nu(0,r) = \frac{1}{d} \left(1-(1-r)^d\right).
\end{equation}

\noindent Define the hyperplane intersection model

\begin{equation} \label{intersection_def}
K^{\lambda} = \mathbb{B} \cap \left( \bigcap_{\widetilde{C} \in \widetilde{\cC_{\lambda}}} \mathbb{H}(\widetilde{C}) \right),
\end{equation}

\noindent where $\mathbb{H}(\widetilde{C})$ is the half-space normal to the vector $\widetilde{C}$ and containing $0$, i.e.,

\begin{equation}
\mathbb{H}(\widetilde{C}) = \widetilde{C} + \{x \in \mathbb{R}^d: x \cdot \widetilde{C} < 0 \}.
\end{equation}

\noindent By definition, $K^\lambda = \mathbb{B}$ if $\widetilde{\cC_\lambda} = \emptyset$. The choice of $\nu$ is special, because it makes $K^\lambda$ asymptotically identical to the Boolean intersection model $I^\lambda$ with copies of $\mathbb{B}$ defined earlier -- the difference is that the copies of the unit disk have been exchanged for half spaces. Indeed, placing a single disk $\mathbb{B}$ with center $Z\Theta$ is like placing the half space $\mathbb{H}((Z-1)\Theta)$;
the difference is asymptotically negligible as $\lambda \to \infty$. The error is quantified precisely in terms of the Hausdorff distance in Section \ref{sec:lemmas}. In this section we work with both $K^\lambda$ and $I^\lambda$ simultaneously: the same ideas can be used to analyze both models, with slightly different details.

For any $\theta \in \mathbb{S}$, define the `radius in the $\theta$ direction:'

\begin{equation} \label{radius_def}
R^{\lambda}(\theta) = \sup \{t \in (0,1): t \theta \in I^\lambda\}, \hspace{10pt}Q^{\lambda}(\theta) = \sup \{t \in (0,1): t \theta \in K^\lambda\}.
\end{equation}

\noindent The collections $\{R^\lambda(\theta): \theta \in \mathbb{S}\}$ and $\{Q^{\lambda}(\theta): \theta \in \mathbb{S}\}$ each consist of identically distributed but not independent variables: $R^\lambda(\theta)$ and $R^\lambda(\theta')$ are not independent for $\theta \neq \theta'$ (except for the special case $\theta' = -\theta$). We simply write $R = R^{\lambda}$ or $Q = Q^{\lambda}$ for the distribution of any one of the $R^{\lambda}(\theta)$ or $Q^{\lambda}(\theta)$, respectively. For $r \in (0,1)$, denote by $F(r)$ and $G(r)$ the probabilities

\begin{equation}\label{Gmu_def}
F(r) = \p(r e_1 \notin B), \hspace{10pt} G(r) = \p(r e_1 \notin H),
\end{equation}

\noindent in which $B=\mathbb{B} + C$, where $C$ is a uniformly distributed point in $\mathbb{B}$, $H$ is an independent copy of one of the random hyperplanes $\mathbb{H}(\widetilde{C}), \widetilde{C} \in \widetilde{\cC_\lambda}$, and $\vec{e}_1 = ( 1, 0, \ldots, 0) \in \R^d.$ The functions $F$ and $G$ are closely connected to the distributions of $R$ and $Q$ respectively, since for $r \in (0,1)$,

\begin{equation} R > r \iff r \vec{e}_1 \in \mathbb{B} + C \text{ for every } C \in \cC_\lambda, \end{equation}

\noindent and similarly for $Q$. Recall that $\omega_d$ denotes the volume of the unit ball $\mathbb{B}$ in $\mathbb{R}^d$. We have the following formulas for $F$ and $G$:

\begin{proposition}\label{G_formula} Consider the intersection boolean models $I^\lambda$ and $K^\lambda$ in dimension $d \geq 1$. \begin{itemize} \item[i.] Let $F$ be as in (\ref{Gmu_def}). For any $r \in (0,1)$,

\begin{align} F(r) &= \frac{1}{\omega_d} \left| \mathbb{B} \setminus (r \vec{e_1} + \mathbb{B}) \right| \\
& = \frac{\omega_{d-1}}{\omega_d} r - O(r^3) \text{ as } r \to 0,
\end{align}

  \begin{center}
	\begin{figure}
		\hspace{43pt} \includegraphics[scale=.35]{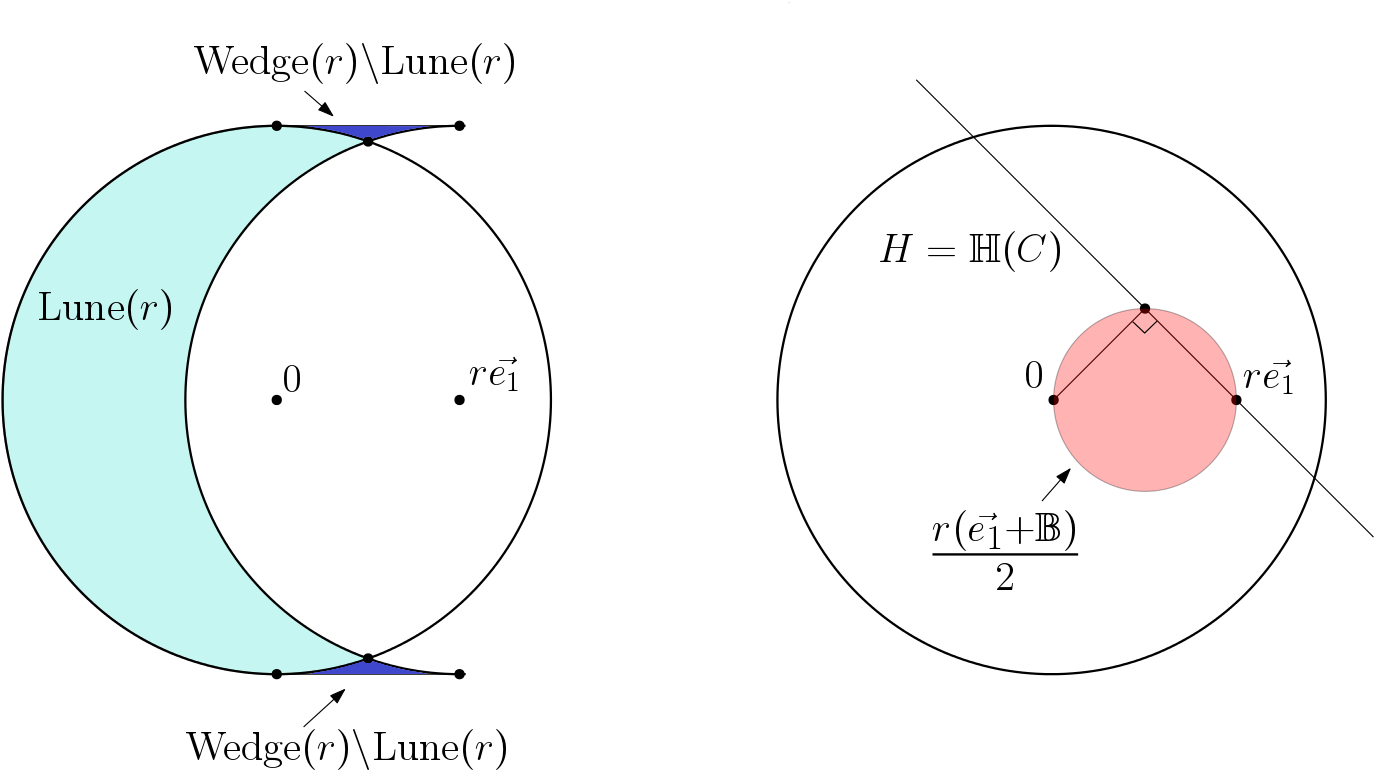}
		\caption{The left diagram shows $F(r)$ as the area of the teal set. Note that if any disk center lies in the lune, then $I^\lambda$ does not contain the point $(r,0)$. The error between the Lune and Wedge sets is shown in blue: its area is negligible, and this gives a quick approximation of the area of the lune. The analogous set for the hyperplane model is shown on the right: $G(r)$ is the mass placed on the red set by the probability measure $\frac{1}{S_d} \nu \times \sigma$. The fact that the red set is a circle follows from elementary geometry: the angle in a semicircle is a right angle.}
	\end{figure}
  \end{center}

\noindent where $| \cdot | $ denotes the volume (Lebesgue measure) in $\mathbb{R}^d$.

\item[ii.] Let $G$ be as in (\ref{Gmu_def}), and write $\mathbb{S}^+ = \{\theta \in \mathbb{S}: \langle \theta, \vec{e_1} \rangle > 0 \}$. For any $r \in (0,1)$,

\begin{align} \label{G_integral} G(r) &=   \frac{1}{\omega_d} \int_{\mathbb{S}^+} \nu(0, r \langle \vec{e}_1, \theta \rangle )\, d\sigma(\theta)\\
& =\nu \times \sigma \left(\frac{r(\vec{e}_1+\mathbb{B})}{2}\right) \\
& =  \frac{(d-1)\omega_{d-1}}{d \omega_{d}} \sum_{k=1}^d (-1)^{1+k} \binom{d}{k} \frac{\Gamma(\frac{d-1}{2}) \Gamma(\frac{k+1}{2})}{2\Gamma(\frac{k+d}{2})}r^k \\
& = \frac{\omega_{d-1}}{\omega_d} r- O(r^2) \text{ as } r \to 0.\end{align}

\end{itemize}
 \end{proposition}

Note that $F$ and $G$ have the same linear approximation near $0$: this is a reflection of the fact that the unit ball has smooth boundary, and hence locally is a hyperplane.

\begin{proof} \itemize \item[$i.$] Write $B = C + \mathbb{B}$ where $C$ is uniform over $\mathbb{B}$, and note that for $r \in (0,1)$,

\begin{equation} r \vec{e_1} \notin B \iff \left\lVert C - r \vec{e_1} \right\rVert > 1 \iff C \in \mathbb{B} \setminus (r \vec{e_1} + \mathbb{B}). \end{equation}

The first equality follows immediately. The shape Lune$(r) = \mathbb{B} \setminus (r\vec{e_1} + \mathbb{B})$ is very close to the shape

\begin{equation}
\text{Wedge}(r) = \{\theta + t \vec{e_1}: \theta \in \mathbb{S}^-, t \in (0,r)\}
\end{equation}

obtained by shifting half of $\mathbb{S}$ by $r$ along the $\vec{e_1}$ axis, where $\mathbb{S}^- = \{\theta \in \mathbb{S}: \langle \theta, \vec{e_1} \rangle < 0 \}$. The volume of Wedge$(r)$ is

\begin{equation}
|\text{Wedge}(r)| = \omega_{d-1} r.
\end{equation}

Note that Lune$(r) = \text{Wedge}(r) \cap \mathbb{B}$. % (The error occurs near the equator set $E = \{\theta \in \mathbb{S}: \langle \theta, \vec{e_1} \rangle = 0\}$.)
The ``error" set $\text{Wedge}(r)\setminus\text{Lune}(r)$ is contained in a region with triangular cross sections: in dimension $d = 2$, it has the form

\begin{align}\text{Wedge}(r) \setminus \text{Lune}(r)  = \{(x, y) \in \mathbb{R}^2: x \in (0,r), y \in (\sqrt{1-z^2}, 1)\},
\end{align}
where $z=\min(x,r-x)$ (see Figure 1). In arbitrary dimension $d$, the set $\text{Wedge}(r)\setminus\text{Lune}(r)$ is contained in the region formed by rotating the same region about the $e_1$ axis (i.e., through the copy of the $(d-2)$-dimensional unit sphere lying in the hyperplane $x^1 = 0$). Using the expansion $\sqrt{1-s^2} = 1 - \frac{s^2}{2} + O(s^4)$, it follows immediately that

\begin{equation}
|\text{Wedge}(r) \setminus \text{Lune}(r) | \leq (d-1)\omega_{d-1} \frac{r^3}{16} + O(r^5),
\end{equation}

as desired.

\item[$ii.$]Write $C = P\Theta$, where $P \in [0,1]$ is distributed according to $\nu$, and $\Theta \in \mathbb{S}$ is distributed according to $\sigma$. The corresponding hyperplane, pinned at $C$, is given by

\begin{equation}
H = \mathbb{H}(C) = \{x: \langle x, \Theta \rangle \leq P\},
\end{equation}

\noindent and thus

\begin{equation}
\p(r \vec{e}_1 \notin H) = \p(P < \langle r\vec{e}_1 , \Theta\rangle). %= \p(0 \leq P <  r \langle \vec{e}_1 ,\Theta\rangle).
\end{equation}

The first equality follows immediately by integrating. Moreover,

\begin{equation}
P<r\langle \vec{e}_1 ,\Theta\rangle \iff \left\lVert P\Theta-\frac{r}{2}\vec{e}_1 \right\rVert < \frac{r}{2} \iff
C = P\Theta \in \frac{r(\vec{e}_1+\mathbb{B})}{2},
\end{equation}

proving the second equality. To obtain the explicit sum formula, we can evaluate the integral directly by expanding the measure $\nu$ as

\begin{equation}
\nu(0,r) = \frac{1}{d} \left(1-(1-r)^d \right) = \frac{1}{d} \sum_{k=1}^d (-1)^{1+k} \binom{d}{k} r^k.
\end{equation}

Thus for $d \geq 2$,

\begin{align} G(r) &= \frac{1}{\omega_d} \int_{\mathbb{S}^+} \nu(0,r \langle \vec{e_1}, \theta \rangle) \, d\sigma(\theta) \\
& = \frac{(d-1)\omega_{d-1}}{d \omega_{d}} \int_0^{\pi/2} \int_{\mathbb{S}^{d-1}} \sum_{k=1}^d (-1)^{1+k} \binom{d}{k} r^k \cos^k \alpha \sin^{d-2} \alpha \, d\phi \, d\alpha \\
& = \frac{(d-1)\omega_{d-1}}{d \omega_{d}} \sum_{k=1}^d (-1)^{1+k} \binom{d}{k} \left( \int_0^{\pi/2} \cos^k \alpha \, \sin^{d-2}\alpha \, d\alpha \right) r^k.
\end{align}

(This integral isn't valid when $d = 1$: in that case, $G(r) = F(r) = 2r.$) The $d\alpha$ integral evaluates to

\begin{equation}
\int_0^{\pi/2} \cos^k \alpha \, \sin^{d-2}\alpha \, d\alpha = \frac{\Gamma(\frac{d-1}{2}) \Gamma(\frac{k+1}{2})}{2\Gamma(\frac{k+d}{2})},
\end{equation}

and the coefficient of the linear term ($k = 1$) is $\frac{(d-1)\omega_{d-1}}{d \omega_d} \cdot d \cdot \frac{\Gamma(\frac{d-1}{2})}{2 \Gamma(\frac{d+1}{2})} = \frac{\omega_{d-1}}{\omega_d}$.

\end{proof}

For example, in dimension $d = 2$, an explicit computation -- namely, for the area of intersection of two circles -- yields

\begin{equation} \label{F_dim2formula}
F(r) = 1-\frac{1}{\pi} \left( 2\arccos(r/2) - r \sqrt{1-r^2/4}\right) = \frac{2r}{\pi} -\frac{r^3}{12\pi} + \cdots
\end{equation}

Directly evaluating the integral (\ref{G_integral}) gives

\begin{equation}
G(r) = \frac{1}{\pi} \int_0^{\pi/2} (2r \cos \theta - r^2 \cos^2 \theta) \, d\theta = \frac{2r}{\pi} - \frac{r^2}{4}.
\end{equation}

Recall that $R$ and $Q$ depend on $d$ via the underlying dimension of the boolean intersection models $I^\lambda$ and $K^\lambda$. The distributions of $R^\lambda$ and $Q^\lambda$ converge to exponential distributions as $\lambda \to \infty$:

\begin{proposition} \label{radius_distrib}

For any $d \geq 1$, as $\lambda \to \infty$,

\begin{equation} \lambda \omega_d F(R^\lambda) \to \text{\normalfont{Exp}}(1),\end{equation}

and

\begin{equation} \lambda \omega_d G(Q^\lambda) \to \text{\normalfont{Exp}}(1),\end{equation}

where the convergence is in distribution.

\end{proposition}

\begin{proof} For any $r \in (0,1)$, since $I^{\lambda}$ is the intersection of Poisson$(\lambda \omega_d)$ many copies of $\mathbb{B}$,

\begin{equation}
 \p(R > r) = \p(r \vec{e}_1 \in I^{\lambda}) = \exp(-\lambda \omega_d F(r)).
\end{equation}

Similarly,

\begin{equation}
\p(Q > r) = \exp(-\lambda \omega_d G(r)).
\end{equation}

\noindent Fix $z \in (0, \lambda \omega_d)$. Note that $F$ and $G$ are strictly increasing on $[0,1]$, and hence invertible. Thus we can write

\begin{align} \p(\lambda \omega_d F(R) > z) &= \p(R > F^{-1}(z / \lambda \omega_d )) \\
& = \exp(-\lambda \omega_d F(F^{-1}(z/\lambda \omega_d ))) \\
& = \exp(-z),
\end{align}

\noindent and similarly for $G(Q)$. Taking $\lambda \to \infty$ finishes the proof.
\end{proof}

\begin{corollary} \label{rad_dist_cor} The exact distributions of $\lambda \omega_d F(R)$ and $\lambda \omega_d G(Q)$ are given by

\begin{equation}\label{radius_exact}
\p(\lambda \omega_d F(R) > z) = \p(\lambda \omega_d G(Q) > z) = \begin{cases} e^{-z}, &0 \leq z < \lambda \omega_d \\ 0, &z > \lambda \omega_d. \end{cases}
\end{equation}

The collections $\{\lambda F(R)\}_{\lambda}$ and $\{\lambda G(Q)\}_\lambda$ are uniformly integrable.
\end{corollary}

%The atom at $z = \lambda$ corresponds to the events $R = 1$ or $Q = 1$, i.e., when the Poisson point process has no points, which occurs with probability $\exp(-\lambda)$.

Proposition \ref{radius_distrib} leads to a quick calculation for the expected area of $I^{\lambda}$ and $K^\lambda$ via the following proposition, which we state for general random sets. For a continuous function $f: \mathbb{S} \to [0,1]$, let

\begin{equation}
U(f) = \{r \theta: r \in [0, f(\theta)], \theta \in \mathbb{S}\}
\end{equation}

\noindent be the set with `radius' $f(\theta)$ in the direction $\theta$. Note that by definition, $I^\lambda = U(R^\lambda)$ and $K^\lambda = U(Q^\lambda)$, where $R^\lambda$ and $Q^\lambda$ are viewed as random functions on $\mathbb{S}$ as in (\ref{radius_def}).
In general, if $A\subset \R^d$ is closed and star-shaped about the origin, then the function $f_A: \mathbb{S} \to [0,1]$ given by

\begin{equation}
 f_A(\theta) = \sup\{t \in (0,1): t \theta \in A\}
\end{equation}

\noindent satisfies $A = U(f_A)$.

A random closed set $X$ is \emph{rotationally invariant} if for any $\Phi \in \text{SO}(d)$, $X$ has the same distribution as $\Phi(X) = \{\Phi(x): x \in X\}$.

\begin{proposition} \label{area_calc} Let $X$ be any random closed set contained in $\mathbb{B}$ that is almost surely star-shaped and rotationally invariant. Let $S: \mathbb{S} \to [0,1]$ be the random function, defined on the same probability space as $X$, satisfying $X = U(S)$. Then

\begin{equation}
\mathbb{E}|X| = \omega_d \E \left[S(\vec{e_1})^d\right].
\end{equation}

\end{proposition}

\begin{proof} The volume of any closed, star-shaped set $A$ with $A = U(f)$, for a continuous function $f: \mathbb{S} \to [0,1]$, is given by

\begin{equation}
|A| = \frac{1}{d} \int_{\mathbb{S}} f(\theta)^d \, d\sigma(\theta).
\end{equation}

\noindent See Appendix \ref{appB} for a proof. Since $X$ is rotationally invariant, $\E\left[S(\theta)^d \right]$ does not depend on $\theta$. Thus,

\begin{align} \mathbb{E}|X| &=  \mathbb{E} \left[\frac{1}{d}\int_{\mathbb{S}} S(\theta)^d \, d\sigma(\theta) \right] \\
& = \frac{1}{d} \int_{\mathbb{S}} \E\left[S(\theta)^d\right] \, d\sigma(\theta) \\
& = \frac{1}{d} \int_{\mathbb{S}} \E\left[S(\vec{e_1})^d\right] \, d\sigma(\theta) \\
& = \frac{1}{d} \E\left[S(\vec{e_1})^d\right] \int_{\mathbb{S}} 1 \, d\sigma(\theta)\\
& = \omega_d \E\left[S(\vec{e_1})^d\right].
\end{align}

\noindent Since $X \subset \mathbb{B}$, the expression inside the expectation is almost surely bounded by $\omega_d$, and hence the interchange of integration and expectation is valid by Fubini's theorem.
\end{proof}

\noindent Applying this proposition to $I^\lambda$ and $K^\lambda$ gives:

\begin{corollary} \label{area_calc_cor} For any $d \geq 2$ and $\lambda > 0$,

\begin{equation} \mathbb{E}\left|I^\lambda\right| = \omega_d \mathbb{E}(R^{\lambda})^d, \hspace{10pt} \mathbb{E}\left| K^{\lambda} \right| = \omega_d \mathbb{E}(Q^\lambda)^d. \end{equation}

\end{corollary}

\noindent We are now ready to compute $\E |I^\lambda|$.

\begin{proposition} \label{limit_area} $\lim_{\lambda \to \infty} \lambda^d \E |I^\lambda| = \lim_{\lambda \to \infty} \lambda^d \E |K^\lambda| = \frac{d! \omega_d}{\omega_{d-1}^d}$. \end{proposition}

\begin{proof} Here we carry out the details only for $I$; the analysis for $K$ is similar.
We aim to show that $F(R^{\lambda})$ is approximately equal to its linear term when $\lambda$ is large.
Let $c>0$ be any constant. By Proposition \ref{radius_distrib} and the definition of $F$,
\begin{align} \p(R^\lambda > c \lambda^{-1/2}) &= \p(c \lambda^{-1/2} \vec{e}_1 \in I^\lambda) \\
& = \exp(- \lambda \omega_d F(c \lambda^{-1/2})) \\
& \leq \exp\left(- \lambda \omega_{d-1} (c \lambda^{-1/2} - O(\lambda^{-3/2}))\right) \\
& = \exp(-O(\sqrt{\lambda})).
\end{align}
Proposition \ref{G_formula} now implies that for some constant $C > 0$,
\begin{equation} \p\left(\left|F(R^\lambda) - \frac{\omega_{d-1}}{\omega_d} R^\lambda\right| \geq C\lambda^{-3/2}\right)\to 0. \end{equation}
Consequently, for any $\eps > 0$,
\begin{equation} \label{lin_rad_approx} \p\left(\left|\lambda F(R^\lambda) - \lambda\frac{\omega_{d-1}}{\omega_{d}} R^\lambda\right|\geq\eps\right)\to 0. \end{equation}
\noindent Since convergence in probability implies convergence in distribution, Proposition \ref{radius_distrib} gives us that
\begin{equation} \lambda \omega_d \frac{\omega_{d-1}}{\omega_{d}} R^\lambda \to_d \text{\normalfont{Exp}}(1) \text{ as } \lambda \to \infty. \end{equation}

\noindent By Corollary \ref{rad_dist_cor}, the collection $\{\lambda F(R^\lambda)\}_{\lambda > 0}$ is uniformly integrable. By Proposition \ref{G_formula}, and since $R^\lambda \to 0$ a.s. as $\lambda \to \infty$, $R^{\lambda} \leq C \cdot F(R^\lambda)$ almost surely for $\lambda$ sufficiently large and some constant $C$. 
It follows that the collection $\{\lambda R^\lambda\}_{\lambda > 0}$ is also uniformly integrable. Together, convergence in distribution and uniform integrability imply convergence of moments (see Theorem 25.12 of~\cite{bill}, for example). Thus, as $\lambda \to \infty$,
\begin{equation} \lambda^d \E( R^\lambda)^d \to \left(\omega_d \cdot \frac{\omega_{d-1}}{\omega_{d}}\right)^{-d} \E[\text{Exp}(1)^d] = \frac{d!}{\omega_{d-1}^d}. \end{equation}

\noindent Combining this with Proposition \ref{area_calc} yields the result.

\end{proof}

%%%%%%%%%%%%%%%%%%%%%%%%%%%%%%%%
\subsection{Formulas for general intensity measures} \label{general_intmodel}
%%%%%%%%%%%%%%%%%%%%%%%%%%%%%%%%

We now give a generalization of these ideas for a wider class of intersection models, where the underlying Poisson process has arbitrary intensity measure, and the sets comprising the intersection are arbitrary. There are many possible ways to generalize this model: we chose a way that yields examples with the same flavor as $I^\lambda$ and $K^\lambda$, but with potentially different asymptotic behavior.

Consider a general Poisson point process $\cC_{\mu, \lambda}$ on $[0,1] \times \mathbb{S}$ with intensity measure $\lambda(\mu \times \sigma)$, where $\lambda > 0$, $\sigma$ is surface-area measure on $\mathbb{S}$, and $\mu$ is a probability measure on $[0,1]$.
For each $x \in \mathbb{B}, x \neq 0$, let $\text{SO}_x(d)$ denote the coset of the subgroup $\text{SO}_{\vec{e_1}}(d)$ of $\text{SO}(d)$ given by

\begin{equation} \text{SO}_x(d) = \left\{\Phi \in \text{SO}(d): \Phi(\vec{e_1}) = \frac{x}{\lVert x \rVert}\right\}. \end{equation}

\noindent Let $A \subset \R^d$ be a fixed closed convex set such that $\mathbb{B} \subset A$ and

\begin{equation}
\partial A \cap \partial \mathbb{B} \neq \emptyset.
\end{equation}

\noindent This assumption ensures that the intersection model almost surely 1) is non-empty for every $\lambda > 0$ and 2) converges to the set $\{0\}$ as $\lambda \to \infty$.

Note that $\text{SO}_{\vec{e_1}}(d)$ is a compact group, and hence $\text{SO}_x(d)$ has a unique Haar probability measure $\kappa_x$. For each $C \in \cC_{\mu, \lambda},$ sample a random element $\Phi_C \in \text{SO}_C(d)$ from $\kappa_C$. The generalized intersection model $I^{\mu, \lambda}_A$ is:

\begin{equation}
I^{\mu, \lambda}_A = \mathbb{B} \cap \left(\bigcap_{C \in \cC_{\mu, \lambda}} \Phi_C(A) + C \right),
\end{equation}

\noindent In words, $\Phi_C(A)$ is a copy of $A$ that we give a random spin and pin at a random point; and $I^{\mu, \lambda}_A$ is the intersection of a Poisson number of such independent copies. Clearly $I^{\mu, \lambda}_A$ is rotationally invariant in distribution.

\noindent Let $\rho$ and $\widehat{\nu}$ denote the probability measures $\rho(0,r) = r^d$ and $\widehat{\nu}(0,r) = 1 - (1-r)^d$ for $r \in (0,1)$. The boolean intersection models $I^\lambda$ and $K^\lambda$ are special cases of the general intersection model with respect to these measures, i.e.,

\begin{equation}
I^\lambda = I^{\rho, \lambda}_{\mathbb{B}}, \hspace{10pt} K^{\lambda} = I^{\widehat{\nu}, \lambda}_{\mathbb{H}},
\end{equation}

\noindent where $\mathbb{H} = \{x \in \R^d: x \cdot \vec{e_1} < 1\}$. Define the auxiliary function $F^\mu_A$ analogous to $F$ in (\ref{Gmu_def}): let $A^\mu = \Phi_C(A) + C$ be one of the randomly rotated and shifted copies of $A$ comprising the intersection model $I^{\mu, \lambda}_A$
where $C$ is drawn from the probability measure $\frac{1}{d \omega_d} \mu \times \sigma$, and let

\begin{equation}
F^{\mu}_A(r) = \p(r \vec{e}_1 \notin A^\mu).
\end{equation}

\noindent There is a simple generic formula in terms of $F^{\mu}_A$ for the expected volume of the intersection.

\begin{fact} \label{general_expectation} For any $\lambda >  0$,

\begin{equation}
\mathbb{E}\left|I^{\mu, \lambda}_A \right| = d \omega_d \int_0^\infty  r^{d-1} \exp(- \lambda \omega_d F^\mu_A(r)) \, dr.
\end{equation}
\end{fact}

\noindent The proof of Fact \ref{general_expectation} is a straightforward consequence of a well known formula for the volume of a random set (see for example~\cite{robbins}).

\begin{proof} Write $I = I^{\mu, \lambda}_A$. By Tonelli's theorem,
\begin{align} \E |I| &= \E \int_{\mathbb{R}^d} 1\{x \in I\} \, dx \\
& = \int_{\mathbb{R}^d} \p(x \in I) \, dx\\
& = \int_{\mathbb{R}^d} \exp(- \lambda \omega_d F^\mu_A(||x||))\, dx\\
&= \int_{\mathbb{S}} \int_{0}^1 \exp(-\lambda \omega_d F^\mu_A(r)) r^{d-1} \, dr\, d \sigma(\theta) \\
& = d \omega_d  \int_0^1  r^{d-1} \exp(-\lambda \omega_d F^\mu_A(r)) \, dr ,
\end{align}
as desired.
\end{proof}

For example, in dimension 2, using the exact equality (\ref{F_dim2formula}) and Fact \ref{general_expectation} gives the exact formula

\begin{equation}
\E|I^\lambda| = 2\pi \int_{0}^1 r \exp\left(-\lambda +\frac{\lambda}{\pi} \left( 2\arccos\left(\frac{r}{2}\right) - r \sqrt{1-\frac{r^2}{4}}\right) \right) \, dr.
\end{equation}

\noindent Unfortunately, this integral doesn't admit a simple asymptotic expansion in $\lambda$. The methods of Proposition \ref{area_calc}, on the other hand, do allow an easy and explicit calculation.

Continuing with the ideas of the previous section, define the radius of the intersection as in (\ref{radius_def}):

\begin{equation}
I^{\mu,\lambda}_A = \sup \{t \in (0,1): t \theta \in I^{\mu,\lambda}\}.
\end{equation}

\noindent The analog of Proposition \ref{radius_distrib} holds in general:

\begin{proposition} \label{general_radius_distrib} Let $A \subset \R^d$ be a fixed closed convex set such that $\mathbb{B} \subset A$ and
\begin{equation}
\partial A \cap \partial \mathbb{B} \neq \emptyset,
\end{equation}
and let $\mu$ be any probability measure on $[0,1]$. As $\lambda \to \infty$,
\begin{equation}
\lambda \omega_d \cdot F^\mu_A(R^{\mu, \lambda}_A) \to_d \text{\normalfont{Exp}}(1).
\end{equation}
\end{proposition}

\noindent The proof is similar to that of Proposition \ref{radius_distrib}. Combining Propositions \ref{area_calc} and \ref{general_radius_distrib} and mimicking the ideas of Proposition \ref{limit_area} gives a general program to compute the asymptotic scaling of the volume of $I^{\mu, \lambda}_A$. Below are two examples of intersection models that can be analyzed using this program.

\noindent \textbf{Example 1} Consider $I^{\rho, \lambda}_{\mathbb{H}}$ in dimension $d = 2$, i.e., the hyperplane intersection model $K^\lambda$, but where the underlying Poisson point process is uniform. In this case,

\begin{equation}
F^{\rho}_{\mathbb{H}}(r) =  \int_0^{\pi/2} r^2 \cos^2 \theta \, d\theta = \frac{\pi r^2}{4}.
\end{equation}

\noindent By Proposition \ref{general_radius_distrib},

\begin{equation}
\frac{\pi^2}{4}\lambda \cdot (R^{\rho, \lambda}_{\mathbb{H}})^2 \to \text{Exp}(1).
\end{equation}

\noindent Thus, by Proposition \ref{area_calc},

\begin{equation}
\lambda \E|I^{\rho, \lambda}_{\mathbb{H}}| = \pi \lambda \E (R^{\rho, \lambda}_{\mathbb{H}})^2 \to \pi \cdot \frac{4}{\pi^2} = \frac{4 }{\pi}.
\end{equation}

\noindent Note the difference in scale: $R^{\rho, \lambda}_{\mathbb{H}}$ is roughly $\lambda^{-1/2}$, while $Q^{\lambda} = R^{\widehat{\nu}, \lambda}_{\mathbb{H}}$ is roughly $\lambda^{-1}$. This limit scaling reflects the behavior of the densities of $\rho$ and $\widehat{\nu}$ near 0.

\noindent \textbf{Example 2} Fix $\beta \in (0,\pi)$, and define the cone set

\begin{equation} \text{Cone}(\vec{e_1}, \beta) = \left\{x \in \mathbb{R}^d: \arccos\left(\frac{\langle x, \vec{e_1} \rangle}{\lVert x \rVert} \right) < \beta \right\}. \end{equation}

\noindent In dimension $d = 2$, one can show that

\begin{equation} F^{\rho}_{\text{Cone}(\vec{e_1}, \beta)}(r) \approx C(\beta) r\text{ as } r \to 0 \end{equation}

\noindent for some constant $C(\beta)$, and thus Propositions \ref{general_radius_distrib} and \ref{area_calc} suggest

\begin{equation} \E|I^{\rho, \lambda}_{\text{Cone}(\vec{e_1}, \beta)}| \approx \frac{2}{\pi C(\beta)^2} \lambda^{-2}. \end{equation}

\noindent Unlike the ball and hyperplane models, however, the cone model does not have the Crofton cell $D_0^1$ as a scaling limit -- see Section \ref{sec:conclusion} for further discussion.

%%%%%%%%%%%%%%%%%%%%%%%%%%%%%%%%%%%%%%%%%%%%
\section{Some geometric and probabilistic lemmas} \label{sec:lemmas} %%%
%%%%%%%%%%%%%%%%%%%%%%%%%%%%%%%%%%%%%%%%%%%%

In this section, we state some basic lemmas needed to prove Theorem \ref{intersection_conv} via the coupling outlined in Section \ref{sec:outline}.
We begin with some notation. For any $\theta \in \mathbb{S}$, let Cone$(\theta, \delta)$ denote the cone set
\begin{equation}
{\rm Cone}(\theta, \delta) = \left\{x \in \mathbb{R}^d: \arccos\left(\frac{\langle x, \theta \rangle}{\lVert x \rVert}\right) < \delta\right\}.
\end{equation}
\noindent Consider the spherical and planar `caps'
\begin{equation}
{\rm Cap}(\theta, \delta)={\rm Cone}(\theta, \delta) \cap \partial \mathbb{B}
\end{equation}
\noindent and
\begin{equation}
{\rm Hyp}(\theta, \delta) = {\rm Cone}(\theta, \delta) \cap \{x: x \cdot \theta = 1\}.
\end{equation}
\noindent Write $d_H$ for Hausdorff distance, defined by
\begin{equation}
d_H(X,Y)=\max\left\{\sup_{x\in X}\inf_{y\in Y}||x-y||,\sup_{y\in Y}\inf_{x\in X}||x-y||\right\}.
\end{equation}
The first fact describes the boundary of the unit ball $\mathbb{B}$.

\begin{center}
	\begin{figure}
		\hspace{180pt} \includegraphics[scale=1]{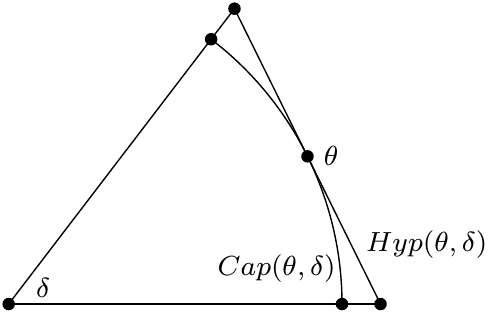}
		\caption{The sets ${\rm Cap}(x,y)$ and ${\rm Hyp}(x,y)$}
	\end{figure}
\end{center}

\begin{fact} \label{trig_lemma} For $\delta$ sufficiently small,
\begin{equation}
d_H({\rm Cap}(\theta, \delta), {\rm Hyp}(\theta, \delta)) =1-\cos\delta \leq \delta^2.
\end{equation}
\end{fact}
%
%\noindent Our second fact gives a bound on the Hausdorff distance between a the intersection model and a shifted copy.
%
%\begin{fact} \label{triangle_fact} Let $X$ denote any intersection of balls (any set in the sample space of $I^\lambda$). Let $x \in \partial X$ be a boundary point lying on the boundary of at least %two of the copies of $\mathbb{B}$ comprising $X$. Let $\alpha$ denote the minimum angle between pairs of copies of $\mathbb{B}$ meeting at $x$, and for $\delta > 0$, let $Y$ be any set obtained by %translating each copy of $\mathbb{B}$ meeting at $x$ by at most $\delta$. Then
%
%\begin{equation} d_H(X, Y) \leq \delta \csc \alpha \leq 2\delta \alpha^{-1}. \end{equation}
%
%\end{fact}
%
%\begin{proof} Easy trigonometry in 2 dimensions, I think the 2-D case is enough to prove the high dimensional one? \end{proof}

\noindent Write ${\rm Sh}^-(\eps)$ for the (inner) spherical shell
\begin{equation}
{\rm Sh}^-(\eps) = \{x \in \mathbb{R}^d: ||x|| \in (1-\eps, 1)\}.
\end{equation}

\noindent Similarly, write ${\rm Sh}^+(\eps)$ for the outer spherical shell
\begin{equation}
{\rm Sh}^-(\eps) = \{x \in \mathbb{R}^d: ||x|| \in (1, 1+\eps)\},
\end{equation}
and set ${\rm Sh}(\eps)={\rm Sh}^-(\eps)\cup\partial\mathbb{B}\cup{\rm Sh}^+(\eps)$.

\begin{fact} \label{sphere_partition} For any $\eps > 0$, there is a partition of ${\rm Sh}^-(\eps)$ into $K=K(d)$ sets $A_i = A_i(\eps),$ $i = 1, \ldots, K$, such that for any $K$ points $x_i \in A_i$, $i = 1, \ldots, K$, the component of the origin in the Boolean tessellation generated by copies of $\partial \mathbb{B}$ centered at the points $x_i$ is contained in $\mathbb{B}_{2\eps} = \{x:||x|| < 2\eps \}$, where $K$ does not depend on $\eps$.
\end{fact}

For instance, when $d=2$, we may take $K(2)=6$, and make the $A_i$ consecutive thickened arcs of length $\pi/3$. To see this, assume for simplicity that all the points $x_i$ lie on the inner boundary of ${\rm Sh}^-(\eps)$, i.e., that
$||x_i||=1-\eps$. Then, with $K(2)=6$, the angle between two consecutive $x_i,x_j$ is at most $2\pi/3$, so that the circles $C_i=x_i+\partial \mathbb{B}$ and $C_j=x_j+\partial \mathbb{B}$ are both tangent to $C_{\eps}=\partial\mathbb{B}_{\eps} = \{x:||x|| = \eps \}$ at points $p_i,p_j$ on $C_{\eps}$ separated by an angle of less than $2\pi/3$. Replacing $C_i,C_j$ by lines $l_i,l_j$ tangent to $C_{\eps}$ at $p_i,p_j$, we see that $l_i$ and $l_j$ make an angle of at least $\pi/3$ enclosing $C_{\eps}$, so that their intersection lies in $\mathbb{B}_{2\eps}$; the same applies to the intersection of $C_i$ and $C_j$.

In higher dimensions, a similar argument applies. Once again, we may assume that the points $x_i$ all satisfy $||x_i||=1-\eps$. The $A_i$ will be thickened cap-like sets partitioning
${\rm Sh}^-(\eps)$ in such a way that the angle $x_iOx'_i$ formed by two points $x_i,x'_i\in A_i$ and the origin $O$ is small. Consequently, the corresponding angle formed by points $x_i,x_j$ in adjacent caps $A_i,A_j$ is also small. The spheres $S_i=x_i+\partial \mathbb{B}$ and $S_j=x_j+\partial \mathbb{B}$ are tangent to $S_{\eps}=\partial\mathbb{B}_{\eps} = \{x:||x|| = \eps \}$ at points $p_i,p_j$ as before; we replace $S_i,S_j$ by hyperplanes $\Pi_i,\Pi_j$, tangent to $S_{\eps}$ at $p_i,p_j$, whose unit normal vectors $n_i,n_j$ are such that $n_i\cdot n_j$ is small. This applies to points in every pair of adjacent caps $A_i,A_j$, so that the vertices of the polytope formed from the hyperplanes $\Pi_i$ all lie inside $\mathbb{B}_{2\eps}$.

Upper bounds on $K(d)$ can be obtained from classical results on covering the surface of a high-dimensional unit ball with small spherical caps; see, for instance, Theorem 6.3.1 of~\cite{Bor}.

We will also need three standard statistical estimates: one regarding the coupon collector process, and two related to Poisson random variables. The `coupon collector' refers to the following discrete time process. Given a finite collection of coupons, say $\{1, 2, \ldots, K\} = [K]$, we generate an iid sequence $(X_t)_{t \leq T}$ of $[K]$-valued random variables from any discrete distribution (not necessarily uniform). The process concludes at the stopping time $T$ when every coupon has been collected at least once:

\begin{equation}
T = \inf\{t \in \N: [K] \subset \{X_1, X_2, \ldots, X_t\}\}.
\end{equation}
Very precise results about the distribution of $T$ are known; see, for instance,~\cite{ER}. Here, we use a simple bound that is convenient for our purposes.

\begin{fact} \label{coupon_fact} Consider any coupon collector process on $K$ coupons, for $K=K(d)$ as in Fact \ref{sphere_partition}. Let $T$ denote the time it takes to collect all $K$ coupons, and let $a_*$ denote the minimum probability of selecting any of the coupons. Then for any $\lambda > 0$,

\begin{equation} \label{coupon_prob}
\p(T > \log \lambda) < K \lambda^{\log(1 - a_*)}.
\end{equation}

\end{fact}

We will apply Fact \ref{coupon_fact} in the following way: we collect coupon $i$ when a point of a Poisson process lands in a region $A_i$. Note that, since the sets $A_i$ appearing in Fact \ref{sphere_partition} do not necessarily have equal volume, the coupon collector process considered in Fact \ref{coupon_fact} is not necessarily uniform. In any case, we can construct the $A_i$ so that $a_* > 0$, and thus the probability in (\ref{coupon_prob}) goes to 0 as $\lambda \to \infty$.

\begin{proof}[Proof of Fact \ref{coupon_fact}] Let $a_1, a_2, \ldots, a_K$ denote the probabilities of selecting the $K$ coupons respectively, and

\begin{equation} a_* = \min \{a_1, a_2, \ldots, a_K\}. \end{equation}

\noindent Let $V_i(t)$ denote the indicator random variable of the event that coupon $i$ has not been collected by time $t$, i.e.,

\begin{equation}
V_i(t) = {\bf  1}_{\{i \notin \{X_1, X_2, \ldots, X_t\}\}}.
\end{equation}

\noindent Note that

\begin{equation} \E V_i(t) = (1-a_i)^t. \end{equation}

\noindent By Markov's inequality,

\[
\p(T > t)  = \p\left(\sum_{i=1}^K V_i(t) \ge 1\right)
 \leq \sum_{i=1}^K \E V_i(t)
 =  \sum_{i=1}^K (1-a_i)^t
 \leq K(1-a_*)^t.
\]

\noindent Setting $t = \log \lambda$ yields the result.

\end{proof}

\begin{proposition}\label{Poisson_TV} Fix any $\mu, \delta > 0$, and let $N \sim \text{\normalfont{Poi}}(\mu)$, $N' \sim \text{\normalfont{Poi}}(\mu + \delta)$. Then

\begin{equation}
{\rm TV}(N, N') \leq \delta,
\end{equation}

\noindent where ${\rm TV}$ denotes total variation distance between probability distributions.

\end{proposition}

\noindent See for instance~\cite{adell}.

\begin{fact} \label{poi_tail} Let $Y \sim \text{\normalfont{Poi}}(V_\eps \lambda),$ where $V_\eps$ is the volume of {\rm Sh}$(\eps_\lambda)$, and $\eps_\lambda = \lambda^{-1}\log^2 \lambda$. Then
\begin{equation}
\p(Y < \log \lambda) < \lambda^{-1},
\end{equation}
\noindent for sufficiently large $\lambda$.
\end{fact}

\noindent See~\cite{ccanonne}, for example.

\medskip

%%%%%%%%%%%%%%%%%%%%%%%%%%%%%%%%%%%%%%%%%%%%
\section{Coupling the intersection and tessellation models} \label{sec:coupling} %%%
%%%%%%%%%%%%%%%%%%%%%%%%%%%%%%%%%%%%%%%%%%%%

The goal of this section is to prove our main convergence result, Theorem \ref{intersection_conv}. We carry out the argument for the intersection model with balls, by coupling it to the tessellation model with spheres. Under this coupling, the Hausdorff distance between the two models tends to 0 as $\lambda \to \infty$, even after re-scaling both sets by $\lambda$.

The coupling works as follows. Let $\mathcal{X} = \mathcal{X}_\lambda$ be a Poisson point process in $\mathbb{R}^d$ of intensity $\frac{\lambda}{2}$, and consider the set
\begin{equation} \label{tess_def}
\bigcup_{x \in \mathcal{X}} (x + \partial \mathbb{B}).
\end{equation}
\noindent Let $J^\lambda$ be the connected component of the origin in the resulting tessellation. Also, let $I^\lambda$ denote the ball intersection model on $\mathbb{B}$, that is,
\begin{equation}
I^\lambda = \bigcap_{C \in \mathcal{C}_\lambda} (C + \mathbb{B}),
\end{equation}
\noindent where $\mathcal{C}_\lambda$ is a Poisson point process on $\mathbb{B}$ of intensity $\lambda$. The reason that $I^{\lambda}$ comes from a point process of intensity twice that of $J^{\lambda}$ is that only points inside $\mathbb{B}$ contribute to $I^{\lambda}$,
while points both inside and outside $\mathbb{B}$ will contribute to the shape of $J^{\lambda}$. Recalling notation from the previous section, define
\begin{equation}
\mathcal{X}^\eps = \mathcal{X} \cap \mathrm{ Sh}(\eps),
\end{equation}
\noindent and consider the `restricted' model $J^\lambda_{\eps}$, the connected component of the origin in the tessellation induced by
\begin{equation} \label{jeps_def}
\bigcup_{x \in \mathcal{X}^\eps} (x + \partial \mathbb{B}).
\end{equation}
\noindent Also, let $\mathcal{C}_\lambda^\eps = C_\lambda \cap \mathrm{ Sh}(\eps) = C_\lambda \cap \mathrm{ Sh}^-(\eps) = \{C \in \mathcal{C}_\lambda: ||C|| \in (1-\eps, 1)\}$, and define
\begin{equation}
I_\eps^\lambda = \bigcap_{C \in \mathcal{C}_\lambda^\eps} (C_i + \mathbb{B}).
\end{equation}
\noindent For $\eps$ chosen appropriately, the restricted models are not far from the original models. Precisely:

\begin{proposition} \label{boundary_restrict} Set $\eps = \eps_\lambda = \frac{\log^2 \lambda}{\lambda}$. As $\lambda \to \infty$,
\begin{equation}
\p(J^\lambda_{2\eps} = J^\lambda) \to 1 \text{ and } \p(I^\lambda_{2\eps} = I^\lambda) \to 1.
\end{equation}
\noindent Moreover, with probability tending to 1, $J^{\lambda} \subset \mathbb{B}_{2\eps}$ and $I^\lambda \subset \mathbb{B}_{2\eps}$.
\end{proposition}

This is essentially the statement that only points near $\partial \mathbb{B}$ contribute to the component of the origin.

\begin{proof} We will prove the assertion for $J^{\lambda}$; the proof for $I^{\lambda}$ is similar. The main ingredients are Fact \ref{sphere_partition} and Fact \ref{coupon_fact}. Fix any $\lambda > 0$, and let $\{A_i(\eps_\lambda)\}_{i=1}^K$ denote the sets given by Fact \ref{sphere_partition}. Consider the stopping time
\begin{equation}
\lambda^{\text{cov}} = \min \{\eta: \mathcal{X}_\eta \cap A_i(\eps_\lambda) \neq \emptyset, {\rm \ for\ }i = 1, 2, \ldots, K \}.
\end{equation}
\noindent We regard the Poisson point processes $\mathcal{X} = \mathcal{X}_\eta$ as nested, i.e., they are coupled so as to form a non-decreasing sequence of sets as $\eta$ increases. This makes $\lambda^{\text{cov}}$ is a well-defined stopping time.

Fix $\eta > \lambda^{\text{cov}}$. By Fact \ref{sphere_partition}, $J^\eta \subset \mathbb{B}_{2\eps_\lambda}$. It follows that any point $y \in \mathcal{X}_\eta$ $\cap $ Sh$(2\eps_\lambda)^c$ does not contribute to $J^\eta$, i.e., $y + \partial \mathbb{B}$ does not lie on the boundary of $J^\eta$. Thus $J^\eta_{2\eps_\lambda} = J^\eta$ in this case.

To finish the proof, it suffices to show that

\begin{equation}
\p(\lambda^{\text{cov}} < \lambda) \to 1 \text{ as } \lambda \to \infty.
\end{equation}

\noindent For the above event to occur, the point process $\mathcal{X}_\lambda$ must have at least one point in each set $A_i(\eps_\lambda)$. It is enough to have $\log \lambda$ points of $\mathcal{X}_\lambda$ in $\mathcal{X}_\lambda^\eps$ -- then by Fact \ref{coupon_fact}, each $A_i$ will have at least one point with high probability. The number of such points is Poisson with mean $V_\eps \cdot \lambda$. By Fact \ref{poi_tail} and a union bound, and noting that $a_*>0$,
\begin{equation}
\p(\lambda^{\text{cov}} \geq \lambda) \leq K \lambda^{\log(1-a_*)} + \lambda^{-1} \to 0 \text{ as } \lambda \to \infty.
\end{equation}
\end{proof}

Additionally, the two restricted models are close to each other.

\begin{proposition} \label{coupling_prop} The random sets $I^\lambda, I_\eps^\lambda, J^\lambda$ and $J_\eps^\lambda$ can be constructed on the same probability space such that as $\lambda \to \infty$, with $\eps = \eps_\lambda/2 = \frac{\log^2 \lambda}{2 \lambda}$ as in Proposition \ref{boundary_restrict},
\begin{equation}
\p(d_H(\lambda I_\eps^\lambda, \lambda J_\eps^\lambda) \leq C\lambda^{-1/4}) \to 1,
\end{equation}
\noindent for some global constant $C > 0$.
\end{proposition}

\begin{proof} We will give an explicit coupling of the intersection and tessellation models on the same probability space where the two models agree with high probability. The coupling works as follows. We will build the intersection model from the Poisson point process $\mathcal{X}$, so the random sets $J^\lambda$ and $J^\lambda_\eps$ are defined above (see (\ref{jeps_def})), and then construct $I^\lambda$ and $I^\lambda_{\eps}$ as functions of $\mathcal{X}$. Fix any $\lambda > 0$, and partition $\mathcal{X}^\eps = \cX^\eps_{\lambda}$ into two sets
\begin{equation}
\mathcal{X}^\eps = \{x \in \mathcal{X}^\eps: ||x|| > 1 \} \cup \{x \in \mathcal{X}^\eps: ||x|| < 1  \} := \mathcal{X}^{\eps, +} \cup \mathcal{X}^{\eps, -}.
\end{equation}
\noindent We list the points of $\mathcal{X}$ lying in $\mathcal{X}^{\eps, +}$ and $\mathcal{X}^{\eps, -}$ as follows:
\begin{equation}
\mathcal{X}^{\eps, +} = \{X_1, X_2, \ldots, X_{N'}\}, \mathcal{X}^{\eps, -} = \{X_{N'+1}, \ldots, X_N\}.
\end{equation}
\noindent Write $X_i = S_i\Theta_i$ for $i = 1, 2, \ldots, N'$, where $S_i >1$ and $\Theta_i \in \mathbb{S}$, and set
\begin{equation}
C_i' = (S_i - 2) \Theta_i \text{ for } i = 1, 2, \ldots, N',
\end{equation}
\noindent while $C_i' = X_i$ for $i = N'+1, \ldots, N$. Generate further points $C_i', i = N+1, \ldots, Z$ from a Poisson process of intensity $\lambda$ in $\{x \in \mathbb{B}: ||x|| < 1-\eps\}$.
Then we set
\begin{equation}
(I^{\lambda})' = \bigcap_{i=1}^Z (C_i' + \mathbb{B}) \text{ and } (I_{\eps}^\lambda)' = \bigcap_{i=1}^N (C_i' + \mathbb{B}).
\end{equation}
\noindent Note that the set $(I_\eps^{\lambda})'$ doesn't have exactly the same distribution as $I_\eps^\lambda$: because ${\rm Sh}^+(\eps)$ has more volume than ${\rm Sh}^-(\eps)$, $(I_\eps^\lambda)'$ has slightly more points than $I_\eps^\lambda$. Also, since the shift map $s \theta \to (s-2)\theta$ is not an isometry, those points are no longer uniformly distributed over the inner shell. It remains to show that error between these two point processes is small. The two distributions in question can be described as follows:

\begin{enumerate} \item $I^\lambda_\eps$ is formed by sampling a Poisson variable $N$ with mean $\mu = \lambda \omega_d (1-(1-\eps)^d)$, and placing $N$ points on ${\rm Sh}^-(\eps)$ with iid uniform angles in $\mathbb{S}$ and iid radii $R$ distributed as

\begin{equation}
\p(R \leq 1-w) = H(w) = \frac{(1-w)^d-(1-\eps)^d}{1-(1-\eps)^d}, w \in (0, \eps).
\end{equation}

\item $(I_\eps^\lambda)'$ is formed by sampling a Poisson variable $N'$ with mean $\mu' = \frac{\lambda}{2} \omega_d ((1+\eps)^d - (1-\eps)^d)$, and placing $N'$ points on ${\rm Sh}^+(\eps)$ with iid uniform angles in $\mathbb{S}$ and iid radii $R'$ distributed as

\begin{equation}
\p(R' \leq 1-w) = H'(w) = 1 - \frac{(1+w)^d - (1-w)^d}{(1+\eps)^d - (1-\eps)^d}, w \in (0, \eps).
\end{equation}

\end{enumerate}

The difference between the means is

\begin{equation}
\mu' - \mu = \lambda \omega_d\left( \binom{d}{2} \eps^2 + O(\eps^3) \right) = O(\lambda \eps^2). 
\end{equation}

By Proposition \ref{Poisson_TV}, ${\rm TV}(N, N') \leq c \lambda \eps^2$ for some $c > 0$. To couple the points of $I_\eps^\lambda$ and $(I_\eps^\lambda)'$, first note that the random angles can be coupled directly, so it suffices to couple the radii. We rely on an idea from optimal transport, namely the optimal coupling for the Wasserstein-1 distance, rather than total variation. The optimal coupling with respect to the $L^1$ distance is given by the transport map $H' \circ H^{-1}$: that is, given $R$, the random variable $H'(H^{-1}(R))$ on the same probability space is distributed as $R'$, and the average $L^1$ distance between $R$ and $R'$ is minimized for this coupling; see for instance Section 2.1 of~\cite{santa}. Explicit computations with series expansions show that for $u \in (0, \eps)$,

\begin{equation}
u- H'\circ H^{-1}(u) = O(\eps^2).
\end{equation}

\noindent Since $\eps^2 \ll \lambda^{-3/2}$ as $\lambda \to \infty$, for $\lambda$ sufficiently large there exists a coupling between $R$ and $R'$ such that

\begin{equation}\label{shifty}
|R - R'| \leq \lambda^{-3/2} \text{ a.s. as } \lambda \to \infty.
\end{equation}
%
%\noindent Thus, we see that $I_\eps^\lambda$ and $(I_\eps^\lambda)'$ can be coupled so that
%
%\begin{equation}\label{shifty}
%\p(d_H(\lambda I_\eps^\lambda, \lambda (I_\eps^\lambda)') > \lambda^{-1/2}) < c \lambda \eps^2 \to 0 \text{ as } \lambda \to \infty.
%\end{equation}
%
  \begin{center}
	\begin{figure}
		\hspace{40pt} \includegraphics[scale=.35]{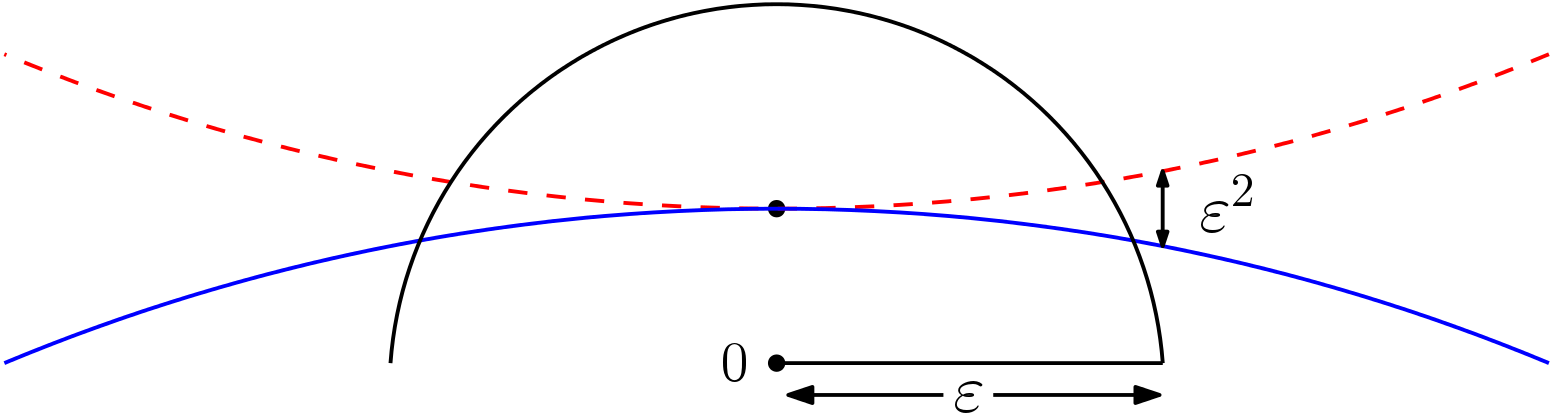}
		\caption{On the event that the intersection model is contained in $\mathbb{B}_\eps$, the difference between $(I^\lambda)'$ and $J^\lambda$ is that some of the copies of $\mathbb{B}$ in the former comprise a concave piece of the boundary (red/dashed), while those in the latter form a convex boundary piece (blue/solid). The difference is small in Hausdorff distance, because the boundary of the ball is differentiable.}
	\end{figure}
  \end{center}
\begin{center}
	\begin{figure}
		\hspace{2in} \includegraphics[scale=.35]{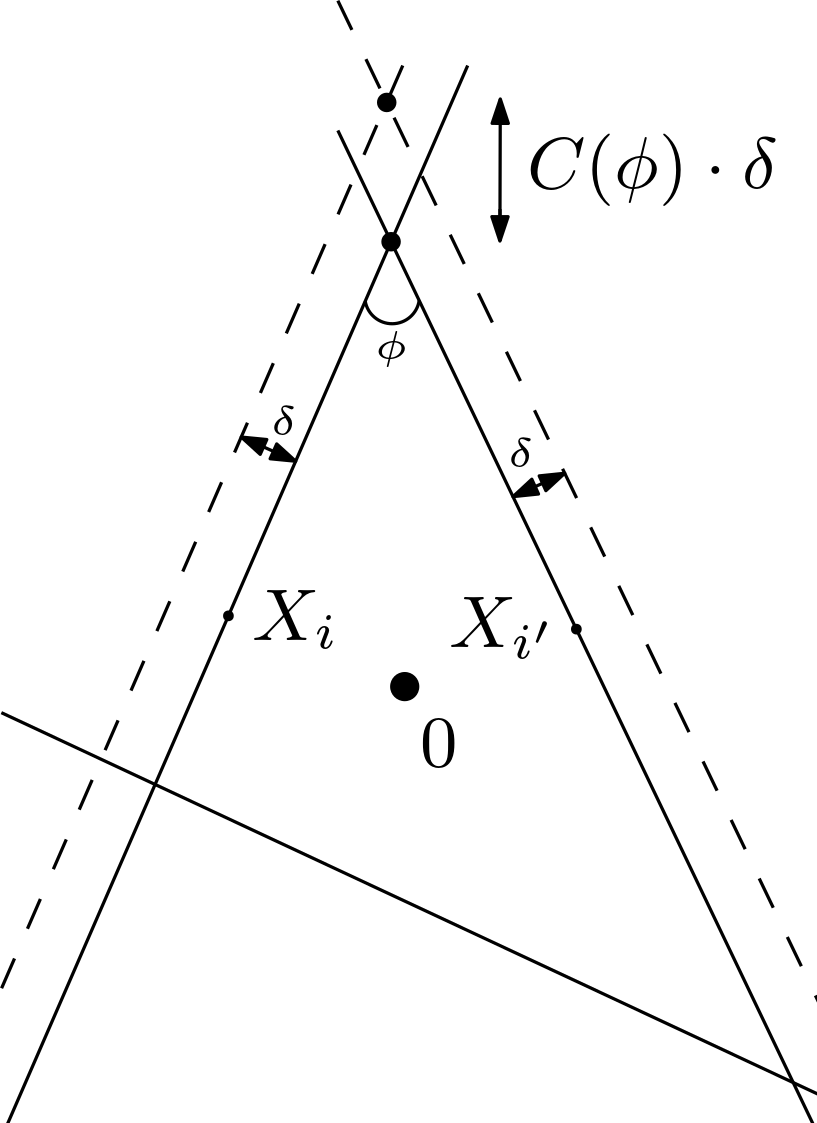}
		\caption{As long as the angles $\phi$ between intersecting hyperplanes are not too small, their $d$-fold intersections move by $O(\delta)$ if the hyperplanes are shifted by at most $\delta$.}
	\end{figure}
  \end{center}
Let us summarize the proof up to this point, ignoring $\eps$ and $\lambda$ for simplicity. We start with the points $X_i$ close to $\partial\mathbb{B}$, which define the tessellation model $J$, a small set near the origin, whose boundary is composed of convex parts (corresponding to $X_i$ lying just inside the unit ball $\mathbb{B}$) and concave parts (coming from $X_i$ lying just outside $\mathbb{B}$). We then shift the $X_i$ lying outside $\mathbb{B}$ to diametrically opposite points $C'_i$; this ``flips" the corresponding spheres, which are now centered at the $C'_i$, so that the resulting set $I'$ is now convex, being composed entirely of convex parts. This last step does not generate a random copy of $I$; we've introduced a slight distortion, in that the density of the points $C_i'$ is not quite uniform. We thus have to shift the $C_i'$ slightly in accordance with (\ref{shifty}), so that the corresponding parts of the boundary of $I'$ also shift a little, until we finally generate a truly random instance of the intersection model $I$.

Facts~\ref{trig_lemma} and (\ref{shifty}) together show that, in carrying out these two steps, the parts of the boundaries of the sets $I$ and $J$ corresponding to a fixed $X_i$ are close (within Hausdorff distance $\lambda^{-3/2}$, before scaling). It remains to show that the sets $I$ and $J$ are themselves close. For this, we must exclude certain undesirable geometric configurations arising when different boundary pieces intersect. Namely, if the angle at which some set of hyperplanes meet is too small, shifting them could move the common intersection by much more than $\delta$. (See Figure 4 for a diagram in dimension $d = 2$.) To complete the proof, we need to show that the probability that these shifts move the boundary of $J$ by more than $\lambda^{1/4}\delta=\lambda^{-5/4}$ tends to zero.  

By Fact~\ref{trig_lemma}, we may replace the spheres $X_i+\partial\mathbb{B}$ by hyperplanes $Y_i\cdot z=\alpha_i\eps$, where \linebreak $Y_i=(||X_i||_2 - 1) \cdot \frac{X_i}{||X_i||_2}$ and $0<\alpha_i< 1$. To further simplify matters, we may upper bound the change in Hausdorff distance by assuming that all $\alpha_i=1$, and that the total number of hyperplanes comprising the intersection is $O(\log^2\lambda)$ (which holds with high probability since the number of hyperplanes is Poisson with mean $C_d\log^2\lambda$). Temporarily rescaling so that $\eps=1$, our task is to show that shifting any $d$ of the hyperplanes $Y_i\cdot z=1$ to $Y_i\cdot z=1+\delta$ moves their intersection by no more than $O(\lambda^{1/4}\delta)$, with high probability. This will happen unless some point $Y_j$ lies within angular distance $\lambda^{-1/4}$ of a great hypersphere defined by some $d-1$ other points $Y_{j_1},\ldots,Y_{j_{d-1}}$. There are $O(\log^{2(d-1)}\lambda)$ such choices for $Y_{j_1},\ldots,Y_{j_{d-1}}$, and for a fixed such choice, the probability that $Y_j$ lies within $\lambda^{-1/4}$ of their great hypersphere is $C_d\lambda^{-1/4}$. Putting the pieces together, the sets $I$ and $J$ are within Hausdorff distance $\lambda^{-5/4}$ with high probability, and so
\begin{equation}
\p(d_H(\lambda I_\eps^\lambda, \lambda J_\eps^\lambda) > C\lambda^{-1/4}) \to 0,
\end{equation}
as desired.

\end{proof}

\noindent The final ingredient is a union bound.

\begin{proof}[Proof of Theorem \ref{intersection_conv}] Combining Propositions \ref{boundary_restrict} and \ref{coupling_prop} gives

\begin{equation} \p(d_H(\lambda J^\lambda, \lambda I^\lambda) > C \lambda \eps^2) \leq \p(J_{2\eps}^\lambda \neq J^\lambda) + \p(I_{2\eps}^\lambda \neq I^\lambda) + \p(d_H(\lambda J_{2\eps}^\lambda, \lambda I_{2\eps}^\lambda) > C\lambda^{-1/4}) \to 0. \end{equation}

\end{proof}

This method could likely be applied to a wider class of intersection models: namely, the ball $\mathbb{B}$ could be replaced by another set satisfying some curvature condition. The difficulty in such a generalization would lie in finding a canonical way to construct a coupling between the intersection model and another model known to have the Crofton cell of a tessellation model as a scaling limit.

%%%%%%%%%%%%%%%%%%%%%%%%%%%%%%%%%%%%%%%%%%%%
\section{The expected volume of $I^{\lambda}$} \label{sec:ccell} %%%
%%%%%%%%%%%%%%%%%%%%%%%%%%%%%%%%%%%%%%%%%%%%

We first sketch Goudsmit's classical calculation for the second moment of the area $A$ of a typical cell in a Poisson line tessellation ${\sf T}$, when $d=2$.

First we recall some basic facts about ${\sf T}$, which can be defined as follows. Let $\mathcal{P}$ be a unit intensity Poisson process on $[0,2\pi)\times\R^+$. We associate the point
$(\theta,\rho)\in\mathcal{P}$ with the line $x\cos\theta+y\sin\theta=\rho$. This yields a translation and rotation invariant measure on families of lines in the plane; ${\sf T}$ is a random instance of this measure. It is not hard to show that the intersections of ${\sf T}$ and another random line form a Poisson process of rate 2, and also that ${\sf T}$ has $\pi$ intersections, and thus $\pi$ cells, per unit area.

Goudsmit in 1945~\cite{Goud} considered tessellations on the surface of a sphere, but we can rephrase his argument in terms of a realization of ${\sf T}$ over a large circular region $R$ of area $N$. We choose two points  $x,y\in R$ uniformly at random and ask for the probability $q_N$ that they lie in the same cell of ${\sf T}$. If $f(t)$ denotes the density function for the area $A$ of a typical cell, then, as $N\to\infty$,
\[
%begin{equation}\ref{eq:nqn}
Nq_N\to\frac{\int_0^{\infty} t^2f(t)\,dt}{\int_0^{\infty}tf(t)\,dt}=\frac{\E(A^2)}{\E(A)},
\]
%end{equation}
since $tf(t)$ is the density function for the area of a cell containing a given point. On the other hand, $x$ and $y$ lie in the same cell of ${\sf T}$ if and only if the line segment $l_{xy}$ joining them does not intersect ${\sf T}$. As noted above, the number of intersections between $l_{xy}$ and ${\sf T}$ is Poisson with rate 2, so there are no such intersections with probability $e^{-2l}$. Consequently,
\[
Nq_N\to \int_0^{\infty}2\pi le^{-2l}\,dl=\frac{\pi}{2}.
\]
Combining these two expressions yields the formula $\E(A^2)=\frac{\pi}{2}\E(A)=\frac12$, which we may rewrite as
\[
\frac{\E(A^2)}{\E(A)^2}=\frac{\pi^2}{2}.
\]

%\todo{(Note the agreement with prop 1.3.)} We still need to normalize to get the agreement with Prop 1.3, which comes in Section 6.1
Goudsmit's argument generalizes to higher dimensions. Write $V_d$ for the volume of a typical cell in a Poisson hyperplane tessellation ${\sf T}_d$ in $\R^d$. This time we will choose a different normalization for the measure $\mu_{d-1}$ on the space $A(d,d-1)$ of affine hyperplanes in $R^d$, in which the measure of hyperplanes meeting the unit ball is 2, i.e., the diameter of the unit ball. Consequently, in the corresponding tessellation ${\sf T}_d$, the expected number of hyperplanes $H$ meeting the unit ball is 2. With $N$ now representing the volume of a large spherical region $R$, and $q_N$ as before, we once again have
\[
%begin{equation}\ref{eq:nqn}
Nq_N\to\frac{\int_0^{\infty} t^2f(t)\,dt}{\int_0^{\infty}tf(t)\,dt}=\frac{\E(V_d^2)}{\E(V_d)}
\]
%end{equation}
as $N\to\infty$. The second part of the calculation requires a slight modification; we now wish to estimate the number $I_d$ of intersections of a line $L$ of length $l$ with ${\sf T}_d$. Once again, $I_d$ has a Poisson distribution, with mean $c_d$ given by Crofton's formula~\cite{SW} (page 172), namely
\[
\frac{1}{l}\int_{A(d,d-1)}|L\cap H|\mu_{d-1}(dH)=\frac{2\omega_{d-1}}{d\omega_d}:=c_d,
\]
where
\[
\omega_d=\frac{\pi^{d/2}}{\Gamma(d/2+1)}
\]
is the volume of the unit ball. Continuing the argument as before, and recalling that $S_d=d\omega_d$ is the surface area of the unit ball, we have
\[
Nq_N\to \int_0^{\infty}S_dl^{d-1}e^{-c_dl}\,dl=\frac{d!}{c_d^d}\omega_d.
\]
As before, this yields
\[
\frac{\E(V_d^2)}{\E(V_d)}=\frac{d!}{c^d_d}\omega_d,
\]
and, since (with this normalization) we have
\[
\E(V_d)=\left(\frac{2}{c_d}\right)^d\frac{1}{\omega_d}
\]
(see~\cite{Math}, pages 156 and 179), we conclude that
\[
\frac{\E(V_d^2)}{\E(V_d)^2}=\frac{d!}{2^d}\omega_d^2,
\]
in agreement with~\cite{Math} (page 179).

\subsection{Intersection model}

In this section, we apply the above results to give an alternative proof of Proposition \ref{volume_conv}. This is essentially contained in Theorem \ref{intersection_conv} and the formulas above; Theorem \ref{intersection_conv} relates the volumes of $I^{\lambda}$ and $D^1_0$, and the above analysis applies to $D^1_0$. The scaling in Theorem \ref{intersection_conv} relies on the scaling in Theorem \ref{CMP}, which was briefly sketched in the introduction. For completeness, we now explain the scaling in more detail, in the context of Theorem \ref{intersection_conv}.

Suppose then that we know that the intersection $I^\lambda = I_d^{\lambda}$ converges to the Crofton cell of a suitably-scaled Poisson hyperplane process, and we wish to determine the scaling.

First, consider the Poisson hyperplane model normalized, as above, so that the expected number of hyperplanes meeting the unit ball is 2. Write $B_{\eps}=\{x:||x|| < \eps \}$, so that the expected number of hyperplanes meeting $B_{\eps}$ is $2\eps$. For this process, the expected volume $\E(V^0_d)$ of the Crofton cell $V^0_d$ is related to the volume of the typical cell $V_d$ as follows:
\[
\E(V^0_d)=\frac{\E(V_d^2)}{\E(V_d)}=\frac{d!}{c^d_d}\omega_d=\frac{d!d^d\omega_d^{d+1}}{2^d\omega_{d-1}^d}.
\]
The last two equalities come from the preceding section. The first equality is proved by calculating the probability that a point chosen uniformly at random in a large region lies in the same cell as the origin (as in the first part of Goudsmit's calculation).

Second, let us choose $\eps=\eps(\lambda)\to 0$ so that the expected number of spheres in the intersection model that meet $B_{\eps}$ tends to infinity. Then the intersection model inside $B_{\eps}$ is well-approximated by a Poisson hyperplane tessellation. The number of intersecting spheres meeting $B_{\eps}$ is the number of points of the underlying Poisson process in the shell
${\rm Sh}^-(\eps) = \{x \in \mathbb{R}^d: ||x|| \in (1-\eps, 1)\}$. This number is asymptotically $\eps\lambda d\omega_d$, i.e., $\lambda d\omega_d/2$ times the ``normalized" value above. Consequently,
as $\lambda\to\infty$,
\[
\E(|I^{\lambda}_d|)=(1+o(1))\E(V^0_d)\left(\frac{2}{\lambda d\omega_d}\right)^d=(1+o(1))\frac{d!\omega_d}{\lambda^d\omega_{d-1}^d},
\]
so that
\[
\lim_{\lambda\to\infty}\lambda^d\E(|I^{\lambda}_d|)\to \frac{d!\omega_d}{\omega_{d-1}^d},
\]
in agreement with Proposition \ref{volume_conv}.

%%%%%%%%%%%%%%%%%%%%%%%%%%%%%%%%%%%%%%%%%%%%
\section{Further questions} \label{sec:conclusion} %%%
%%%%%%%%%%%%%%%%%%%%%%%%%%%%%%%%%%%%%%%%%%%%

We close with some questions related to intersection processes that arose in the course of our research.

\begin{question} What class of intersection models has the Crofton cell of the hyperplane tessellation model $D_0^1$ as a scaling limit? \end{question}

\noindent Some kind of curvature condition is likely necessary: for example, consider the cone intersection model $I^{\rho, \lambda}_{\text{Cone}(\vec{e_1}, \beta)}$ (discussed at the end of section \ref{general_intmodel}). Just like the hyperplane model, the cone model is a polyhedron. However, when $\beta \neq \frac{\pi}{2}$, each cone whose apex lies at the boundary of the intersection will contribute an additional vertex. Back of the envelope calculations show that these vertices remain with positive probability in the limit $\lambda \to \infty$, which strongly suggests that the cone model does not have $D_0^1$ as a scaling limit.

\begin{question} Proposition \ref{general_radius_distrib} shows that the radius $R^{\mu, \lambda}_A$ converges to an exponential distribution under some rescaling. How does the rescaling relate to the geometry of $A$? How does it relate to the behavior of $\mu$? \end{question}

For example, when $A = \mathbb{B}$, only the behavior of the density of $\mu$ near 0 matters in the limit.

\begin{question} Let $S_\lambda$ denote the number of faces of the intersection model $I^\lambda$. Fact \ref{sphere_partition} suggests that $S_\lambda$ is uniformly tight -- i.e. that for any $\delta > 0$ there exists $t > 0$ such that for any $\lambda > 0, \p(S_\lambda > t) < \delta$ -- which would imply the existence of a limiting stationary distribution for $S$.
Can the limiting distribution and transition probabilities for the stochastic process $(S_\lambda)_{\lambda > 0}$ be computed or approximated? \end{question}

A similar question has been answered for the typical Poisson-Voronoi cell and for the Crofton cell of a hyperplane tessellation, in a different limiting context \cite{Calkasides,CS}. The intersection construction of the Crofton cell $D_0^1$ seems ideal to study this question, because it has a natural renewal quality as $\lambda$ increases.

%%%%%%%%%%%%%%%%%%%%%%%%%%%%%%%%%%%%%%%%%%%%
\section{Appendix} \label{sec:appendix} %%%
%%%%%%%%%%%%%%%%%%%%%%%%%%%%%%%%%%%%%%%%%%%%

\subsection{Appendix A} \label{appA}

In this appendix we recall the definition and some basic properties of Poisson point processes, which are used throughout this article.

\begin{definition} Let $\Omega \subset \mathbb{R}^d$ be a Borel set, and $\mu$ a finite Borel measure on $\Omega$. The \textbf{Poisson point process (PPP)} on $\Omega$ with intensity $\mu$, denoted $\cX$, is the finite set consisting of Poisson$(\mu(\Omega))$ many independent points sampled from the probability measure $\frac{1}{\mu(\Omega)} \mu$. \end{definition}

This is a natural, intuitive definition; for a proper introduction to the theory of Poisson point processes, see~\cite{king}.

\begin{fact} Suppose $\cX$ is a PPP on a Borel set $\Omega \subset \mathbb{R}^d$ with intensity meausre $\mu$. \begin{itemize} \item[a.] Let $A \subset \Omega$ be any Borel set. Then

\begin{equation} |A \cap \cX| \sim \text{\normalfont{Poisson}}(\mu(A)). \end{equation}

\item[b.]  If $A, B \subset \Omega$ are disjoint Borel sets, then $\cX \cap A$ and $\cX \cap B$ are independent.
\end{itemize} \end{fact}

An important special case of Fact $8.2a$ is that the probability that a given region $A$ contains no points of $\cX$ is given by:

\begin{equation} \p(A \cap \cX = \emptyset) = \exp(- \mu(A)). \end{equation}

\subsection{Appendix B} \label{appB}

In this appendix we note a formula for finding the volume of a star-shaped domain in $\mathbb{R}^m$ by doing an integration over the $(m-1)$ dimensional sphere.

\begin{fact}
Suppose
$g: \mathbb{S}^{m-1}\to\R^+$ is continuous, and let $A\subset \R^m$ be the star-shaped open set defined by
\begin{equation*}
A = \{x\in\R^m: |x|<g(x/|x|)\}.
\end{equation*}
Then $\text{Vol}(A) = \frac{1}{m}\int_{\mathbb{S}^{m-1}}g^m\,dS$,
where $dS$ is the $(m-1)$-dimensional surface area form on $\mathbb{S}^{m-1}$.
\end{fact}

\begin{proof} The formula comes from integrating in spherical coordinates. Applying Folland's theorem 2.49~\cite{Folland} to the function $f(x) = 1\{x \in A\}$ gives

\begin{align} \text{Vol}(A) &= \int_{\mathbb{R}^d} f(x) \, dx \\
& = \int_{S^{m-1}} \int_{0}^{\infty} f(rx') r^{m-1} \, dr \, d\sigma(x') \\
& = \int_{S^{m-1}} \int_{0}^{g(x')} r^{m-1} \, dr \, d\sigma(x') \\
& = \frac{1}{m} \int_{S^{m-1}} g(x')^m \, d\sigma(x').
\end{align}

\end{proof}

\end{document}